\documentclass[12pt]{article}
\usepackage[utf8]{inputenc}
\usepackage{amssymb,amsmath,amsfonts,amsthm,amscd,latexsym,indentfirst,verbatim,xcolor}
\usepackage[T2A]{fontenc}
\usepackage{geometry}
\def\dbl{\lbrace\kern-3pt\lbrace}
\def\dbr{\rbrace\kern-3pt\rbrace}

\def\End{\textrm{End}}

\def\As{\textrm{As}}
\def\Span{\textrm{Span}}
\def\id{\operatorname{id}}
\def\tr{\textrm{tr}}

\newcommand{\Endrcf}{\operatorname{End}_{\mathrm{rcf}}}

\theoremstyle{plain}
\newtheorem{theorem}{Theorem}[section]
\newtheorem{lemma}[theorem]{Lemma}

\newtheorem*{conjecture*}{Conjecture}

\newtheorem{proposition}[theorem]{Proposition}
\newtheorem{corollary}[theorem]{Corollary}

\theoremstyle{definition}
\newtheorem{definition}[theorem]{Definition}

\newtheorem{remark}[theorem]{Remark}

\begin{document}
\sloppy
\hfill{17B63 (MSC2020)}

\begin{center}
{\Large
Simple double Lie algebras on the Laurent polynomial space}

\smallskip

Vsevolod Gubarev
\end{center}

\begin{abstract}
We construct a simple $\lambda$-double Lie algebra on $k[t,t^{-1}]$ for every nonzero $\lambda\in k$. We then show that the analogous two-sided construction of weight zero is also simple.
Both products admit natural coefficient realizations via row-and-column-finite operators associated with the finitary $\mathbb Z\times\mathbb Z$ matrix algebra.

\end{abstract}

\section{Introduction}

In 2008, M. Van den Bergh introduced~\cite{DoublePoisson} double Poisson algebras on a given associative algebra as a noncommutative analog of Poisson algebra.
The goal behind this notion was to develop noncommutative Poisson geometry.

A~double Lie algebra is a vector space~$V$ endowed with a double bracket
satisfying the double skew-symmetry and double Jacobi identity,
and no associative multiplication on~$V$ is the part of the structure~\cite{Schedler}.
Every double Lie algebra structure defined on a vector space~$V$ can be uniquely extended to
a~double Poisson algebra structure on the free associative algebra $\As\langle V\rangle$.

It is known that double Lie algebras on a finite-dimensional vector space~$V$ are in one-to-one correspondence with skew-symmetric Rota---Baxter operators of weight~0 on the algebra~$\End(V)$~\cite{DoubleLie,DoublePoissonFree,Schedler}.
In~\cite{Double-0}, the correspondence between double Lie algebras and skew-symmetric RB-operators of weight~0 on the matrix algebra was extended to the infinite-dimensional setting.

In~\cite{DoubleLie}, M. Goncharov and P. Kolesnikov proved
that there are no finite-dimensional simple double Lie algebras.
The example of a~countable-dimensional simple double Lie algebra was found in~\cite{Double-0}.

In~\cite{GoncharovGubarev}, the correspondence between double Lie algebras on a finite-dimensional vector space~$V$ with skew-symmetric Rota---Baxter operators of weight~0 on~$\End(V)$ was extended to the case of nonzero weight. Thus, the notion of double Lie algebra of weight~$\lambda$ appeared.
Moreover, it was proved in~\cite{GoncharovGubarev} that every double Lie algebra structure of nonzero weight defined on a vector space~$V$ can be uniquely extended to a~modified double Poisson algebra~\cite{Arthamonov0} on~$\As\langle V\rangle$,
structure with weaker anti-commutativity and Jacobi identity.

In~\cite{GoncharovGubarev}, it was proved that there are no finite-dimensional simple double Lie algebras of nonzero weight.
The purpose of this paper is to give an example of a~simple double Lie algebra of nonzero weight by extending the double bracket defined in~\cite{GoncharovGubarev} on $k[t]$ to the Laurent polynomial space $k[t,t^{-1}]$.

The weight-one product is
\begin{equation}\label{eq:w1}
\dbl f,g\dbr_1(x,y)=y\frac{f(x)g(y)-g(x)f(y)}{y-x},\qquad f,g\in k[t,t^{-1}],
\end{equation}
where $x=t\otimes1$ and $y=1\otimes t$. The weight-zero product is
\begin{equation}\label{eq:w0}
\dbl f,g\dbr_0(x,y)=\frac{f(x)g(y)-g(x)f(y)}{y-x}.
\end{equation}
Thus the two constructions are related by
\begin{equation}\label{eq:relation}
\dbl f,g\dbr_1(x,y)=y\dbl f,g\dbr_0(x,y).
\end{equation}
Although this relation is elementary, the defining identities and the proofs of simplicity are different enough that we give both arguments in full.

Both double brackets~\eqref{eq:w1} and~\eqref{eq:w0} on $k[t]$ appeared in~\cite{GoncharovGubarev,Double-0}
and actually in terms of B\'{e}zout operators in~\cite{OgievetskyPopov}.

The construction of the simple double Lie algebra of nonzero weight was developed with the assistance of Hyra~\cite{Hyra}.

\section{Definitions and elementary lemmas}

Let $W$ be a vector space over $k$. For $u\in W^{\otimes n}$ and a permutation $\sigma$ of the tensor factors, write $u^\sigma$ for the result of applying $\sigma$ to $u$.

\begin{definition}
Let $\lambda\in k$. A $\lambda$-double Lie algebra is a vector space $W$ equipped with a bilinear map $\dbl\cdot,\cdot\dbr\colon W\times W\to W\otimes W$ such that, for all $a,b,c\in W$,
\begin{equation}\label{eq:anticom}
\dbl a,b\dbr+\dbl b,a\dbr^{(12)}=\lambda(a\otimes b-b\otimes a)
\end{equation}
and
\begin{equation}\label{eq:Jacobi}
\dbl a,\dbl b,c\dbr\dbr_L-\dbl b,\dbl a,c\dbr\dbr_R-\dbl\dbl a,b\dbr,c\dbr_L
=-\lambda(b\otimes\dbl a,c\dbr)^{(12)}.
\end{equation}
The three actions on tensor products are determined by
$$
\dbl a,u\otimes v\dbr_L=\dbl a,u\dbr\otimes v,\quad
\dbl a,u\otimes v\dbr_R=u\otimes\dbl a,v\dbr,\quad
\dbl u\otimes v,a\dbr_L=(\dbl u,a\dbr\otimes v)^{(23)}.
$$
For $\lambda=0$ this is an ordinary double Lie algebra.
\end{definition}

\begin{definition}
A subspace $I\subseteq W$ is an ideal if
$\dbl W,I\dbr+\dbl I,W\dbr\subseteq I\otimes W+W\otimes I$.
A $\lambda$-double Lie algebra is simple if its double product is nonzero and its only ideals are~0 and~$W$.
\end{definition}

We shall use the identifications
$V\otimes V=k[x^{\pm1},y^{\pm1}]$, $V^{\otimes3}=k[x^{\pm1},y^{\pm1},z^{\pm1}]$.
For example, $t^a\otimes t^b$ is identified with $x^a y^b$.

\begin{lemma}\label{lem:tensor-kernel}
Let $I$ be a subspace of a vector space $W$, and let $q\colon W\to W/I$ be the quotient map. Then
\begin{equation}\label{eq:kerqq}
\ker(q\otimes q)=I\otimes W+W\otimes I.
\end{equation}
Moreover, if $v\in W$ and $v\otimes v\in I\otimes W+W\otimes I$, then $v\in I$.
\end{lemma}

\begin{proof}
Choose a vector-space complement $C$ such that $W=I\oplus C$. Then
$$
W\otimes W=(I\otimes I)\oplus(I\otimes C)\oplus(C\otimes I)\oplus(C\otimes C).
$$
The restriction of $q\otimes q$ to $C\otimes C$ is an isomorphism onto $(W/I)\otimes(W/I)$, while the other three summands are in its kernel. This proves~\eqref{eq:kerqq}. For the last assertion,~\eqref{eq:kerqq} gives $q(v)\otimes q(v)=0$.
A pure tensor of two nonzero vectors over a field is nonzero. Hence $q(v)=0$, so $v\in I$.
\end{proof}

For a nonzero Laurent polynomial
$f=\sum_{j=a}^b c_jt^j$, $c_ac_b\neq0$,
define its support width by $\operatorname{wd}(f)=b-a$.

The following result is straightforward:

\begin{lemma}\label{lem:support}
Let $I\neq0$ be a subspace of $V$, and choose a nonzero element $f\in I$ of minimal support width $d$. 
If an interval of integral indices has support width strictly smaller than $d$, then the span of the corresponding basis vectors has zero intersection with $I$.
\end{lemma}

\section{The construction of nonzero weight}

Let $e_n=t^n$ for $n\in\mathbb Z$. Define a bilinear map on $V$ by
\begin{equation}\label{eq:defw1}
\dbl f,g\dbr_1(x,y)=y\frac{f(x)g(y)-g(x)f(y)}{y-x}.
\end{equation}
Since $y-x$ divides $f(x)g(y)-g(x)f(y)$ for any $f,g$, the right-hand side belongs to $V\otimes V$.

\begin{proposition}\label{prop:Weight1}
On the basis $(e_n)_{n\in\mathbb Z}$, the double bracket~\eqref{eq:defw1} is
\begin{equation}\label{eq:basis1}
\dbl e_a,e_b\dbr_1=
\begin{cases}
\displaystyle\sum_{j=a}^{b-1} e_j\otimes e_{a+b-j}, & a<b,\\[1.2em]
\displaystyle-\sum_{j=b}^{a-1} e_j\otimes e_{a+b-j}, & a>b,\\[1.2em]
0, & a=b.
\end{cases}
\end{equation}
\end{proposition}

\begin{proof}
If $a<b$, we have
$$
y\frac{x^a y^b-x^b y^a}{y-x}=\sum_{j=a}^{b-1}x^j y^{a+b-j}.
$$
This is the first case of~\eqref{eq:basis1}. Interchanging $a$ and $b$ gives the second case, and the diagonal case is immediate.
\end{proof}

\begin{proposition}
The double bracket~\eqref{eq:defw1} satisfies~\eqref{eq:anticom} with $\lambda=1$.
\end{proposition}

\begin{proof}
Put $N_{f,g}(x,y)=f(x)g(y)-g(x)f(y)$.
After switching the tensor factors in $\dbl g,f\dbr_1$, we get
$$
\dbl g,f\dbr_1(y,x)
=x\frac{g(y)f(x)-f(y)g(x)}{x-y}
=-x\frac{N_{f,g}(x,y)}{y-x}.
$$
Consequently,
$$
\dbl f,g\dbr_1(x,y)+\dbl g,f\dbr_1(y,x)=N_{f,g}(x,y),
$$
which is exactly
$\dbl f,g\dbr_1+\dbl g,f\dbr_1^{(12)}=f\otimes g-g\otimes f$.
\end{proof}

\begin{theorem}
The double bracket~\eqref{eq:defw1} is a double Lie bracket of weight~1. Moreover, the 1-double Lie algebra $(V,\dbl\cdot,\cdot\dbr_1)$ is simple.
\end{theorem}

\begin{proof}
Given a~Laurent polynomial $f$, write $f_x=f(x)$, $f_y=f(y)$, and $f_z=f(z)$, and use analogous notation for $g$ and $h$. Let us compute the three terms on the left-hand side of~\eqref{eq:Jacobi}:
\begin{gather*}
T_1=\frac{y}{y-x}\left(f_x\frac{z(g_yh_z-h_yg_z)}{z-y}-f_y\frac{z(g_xh_z-h_xg_z)}{z-x}\right),\\
T_2=\frac{z}{z-y}\left(g_y\frac{z(f_xh_z-h_xf_z)}{z-x}-g_z\frac{y(f_xh_y-h_xf_y)}{y-x}\right),\\
T_3=\frac{z}{z-x}\left(h_z\frac{y(f_xg_y-g_xf_y)}{y-x}+h_x\frac{y(f_zg_y-g_zf_y)}{z-y}\right).
\end{gather*}
The right-hand side of the Jacobi identity, moved to the left, contributes
$T_4=g_y\dfrac{z(f_xh_z-h_xf_z)}{z-x}$.
Thus the Jacobi residual is $J_1=T_1-T_2-T_3+T_4$.

Let $D=(y-x)(z-x)(z-y)$. Multiplying by $D$ and combining the part of $-T_2$ containing $z^2g_y(f_xh_z-h_xf_z)$ with $T_4$, we obtain
\begin{multline}\label{Jacobi1}
\frac{DJ_1}{yz}
 = f_x(g_yh_z-h_yg_z)(z-x)-f_y(g_xh_z-h_xg_z)(z-y)\\
 -g_y(f_xh_z-h_xf_z)(y-x)+g_z(f_xh_y-h_xf_y)(z-x)\\
-(f_xg_y-g_xf_y)h_z(z-y)-h_x(f_zg_y-g_zf_y)(y-x).
\end{multline}
Collecting the coefficients of $z-x$, $z-y$, and $y-x$ reduces this expression to
$$
(f_xg_yh_z-h_xf_yg_z)((z-x)-(z-y)-(y-x))=0.
$$
Hence $DJ_1=0$. The Laurent polynomial ring is an integral domain and $D\neq0$, so $J_1=0$.
This is precisely~\eqref{eq:Jacobi}.

Let $I\neq0$ be an ideal. Choose
$f=\sum_{r=a}^b c_re_r\in I$, $c_ac_b\neq0$,
of minimal support width $d=b-a$. Suppose first that $d>0$, and put
$$
U=\Span\{e_a,e_{a+1},\dots,e_{b-1}\},\quad
W=\Span\{e_{a+1},e_{a+2},\dots,e_b\}.
$$
Lemma~\ref{lem:support} gives $I\cap U=I\cap W=0$. Therefore the quotient map $q\colon V\to V/I$ is injective on both $U$ and $W$, and $q\otimes q$ is injective on $U\otimes W$.

For $1\le r\le d$, formula~\eqref{eq:basis1} gives
$$
\dbl e_a,e_{a+r}\dbr_1=\sum_{s=0}^{r-1}e_{a+s}\otimes e_{a+r-s}\in U\otimes W.
$$
It follows that $\dbl e_a,f\dbr_1\in U\otimes W$.
This tensor is nonzero. Indeed, its terms whose two indices have sum $2a+d$ are exactly
$c_b\sum_{s=0}^{d-1}e_{a+s}\otimes e_{b-s}$,
and this is a nonzero sum of distinct basis tensors. On the other hand, $f\in I$ and $I$ is an ideal, so
$$
\dbl e_a,f\dbr_1\in I\otimes V+V\otimes I=\ker(q\otimes q).
$$
This contradicts the injectivity of $q\otimes q$ on $U\otimes W$. Therefore $d=0$, and $I$ contains a~basis vector $e_n$.

It remains to propagate this information in both directions. Formula~\eqref{eq:basis1} gives
$$
\dbl e_{n-2},e_n\dbr_1=e_{n-2}\otimes e_n+e_{n-1}\otimes e_{n-1}.
$$
The first summand belongs to $V\otimes I$. Hence $e_{n-1}\otimes e_{n-1}\in I\otimes V+V\otimes I$,
and Lemma~\ref{lem:tensor-kernel} yields $e_{n-1}\in I$. Similarly,
$$
\dbl e_{n+2},e_n\dbr_1=-e_n\otimes e_{n+2}-e_{n+1}\otimes e_{n+1}
$$
implies $e_{n+1}\in I$. Repeating these two arguments gives $e_m\in I$ for every $m\in\mathbb Z$. Thus $I=V$.

Finally, the product is nonzero because $\dbl e_0,e_1\dbr_1=e_0\otimes e_1$. Hence the algebra is simple.
\end{proof}

\begin{corollary}
For every nonzero $\lambda\in k$, the product
$$
\dbl f,g\dbr_\lambda=\lambda y\frac{f(x)g(y)-g(x)f(y)}{y-x}
$$
makes $V$ into a simple $\lambda$-double Lie algebra.
\end{corollary}

\section{The analogous construction of weight zero}
Define
\begin{equation}\label{eq:defw0}
\dbl f,g\dbr_0(x,y)=\frac{f(x)g(y)-g(x)f(y)}{y-x}.
\end{equation}

\begin{proposition}
On the basis $(e_n)_{n\in\mathbb Z}$, the double bracket~\eqref{eq:defw0} is
\begin{equation}\label{eq:basis0}
\dbl e_a,e_b\dbr_0=
\begin{cases}
\displaystyle\sum_{j=a}^{b-1} e_j\otimes e_{a+b-1-j}, & a<b,\\[1.2em]
\displaystyle-\sum_{j=b}^{a-1} e_j\otimes e_{a+b-1-j}, & a>b,\\[1.2em]
0, & a=b.
\end{cases}
\end{equation}
\end{proposition}

\begin{proof}
The proof is analogous to the proof of Proposition~\ref{prop:Weight1}
without the additional factor $y$ occurring in the weight-one product.
\end{proof}

\begin{remark}
Formula~\eqref{eq:basis0} also follows from~\eqref{eq:basis1} and the relation 
$\dbl f,g\dbr_1 = y\dbl f,g\dbr_0$.
\end{remark}

\begin{theorem}
The double bracket~\eqref{eq:defw0} is a double Lie bracket of weight~0. Moreover, the double Lie algebra $(V,\dbl\cdot,\cdot\dbr_0)$ is simple.
\end{theorem}

\begin{proof}
Switching $f$ and $g$ and then switching $x$ and $y$ gives $\dbl g,f\dbr_0(y,x)=-\dbl f,g\dbr_0(x,y)$, so the product is skew-symmetric.

For the Jacobi identity, repeat the computation from the weight-one case after deleting the second-variable factors $y$ and $z$ from every double product. 
If $J_0$ denotes the resulting Jacobi residual, then direct substitution gives
$DJ_0$ equals the right-hand side of~\eqref{Jacobi1}, where $D=(y-x)(z-x)(z-y)$.
Thus $J_0=0$, which is the ordinary double Jacobi identity.

Now, we prove simplicity.
Let $I\neq0$ be an ideal, and choose
$$
f=\sum_{r=a}^b c_re_r\in I,\qquad c_ac_b\neq0,
$$
of minimal support width $d=b-a$. Suppose that $d>0$, and set
$U=\Span\{e_a,e_{a+1},\dots,e_{b-1}\}$.
By Lemma~\ref{lem:support}, $I\cap U=0$. Therefore the quotient map $q\colon V\to V/I$ is injective on $U$, and $q\otimes q$ is injective on $U\otimes U$.

For $1\le r\le d$, formula~\eqref{eq:basis0} gives
$$
\dbl e_a,e_{a+r}\dbr_0=\sum_{s=0}^{r-1}e_{a+s}\otimes e_{a+r-1-s}\in U\otimes U.
$$
Consequently, $\dbl e_a,f\dbr_0\in U\otimes U$.
It is nonzero because its component whose two indices have sum $2a+d-1$ is
$c_b\sum_{s=0}^{d-1}e_{a+s}\otimes e_{b-1-s}\neq0$.
But $f\in I$, so this tensor belongs to $I\otimes V+V\otimes I=\ker(q\otimes q)$, a contradiction. It follows that $d=0$, and therefore $e_n\in I$ for some $n\in\mathbb Z$.

Now
$\dbl e_{n-1},e_n\dbr_0=e_{n-1}\otimes e_{n-1}$.
Hence, $e_{n-1}\in I$. To move in the other direction, use
$$
\dbl e_n,e_{n+3}\dbr_0=e_n\otimes e_{n+2}+e_{n+1}\otimes e_{n+1}+e_{n+2}\otimes e_n.
$$
The first and third summands already lie in $I\otimes V+V\otimes I$. 
Hence the middle summand also lies there, and so $e_{n+1}\in I$. Induction in both directions yields $I=V$.

The product is nonzero since $\dbl e_0,e_1\dbr_0=e_0\otimes e_0$. Thus the double Lie algebra is simple.
\end{proof}

\section{The bilateral matrix realizations}
We now explain precisely in what sense the preceding constructions come from $\mathbb Z\times\mathbb Z$ matrices.

Let $E_{ij}\in\End(V)$ be the matrix unit defined by 
$E_{ij}(e_m)=\delta_{jm}e_i$, $i,j,m\in\mathbb Z$.
Let
$$
I_{\mathbb Z}=\Span_k\{E_{ij}\colon i,j\in\mathbb Z\}
$$
be the finitary matrix algebra. Let $\Endrcf(V)$ denote the algebra of matrices having only finitely many nonzero entries in every row and every column. The matrix units form a~basis of $I_{\mathbb Z}$, but they do not form a~basis of the whole algebra $\End(V)$.

Define a linear map
$P_1\colon I_{\mathbb Z}\to\Endrcf(V)$
on matrix units by
\begin{equation}\label{eq:P1}
P_1(E_{ij})=
\begin{cases}
\displaystyle\sum_{p\ge1}E_{i+p,j+p}, & i\le j,\\[0.8em]
\displaystyle-\sum_{p\ge0}E_{i-p,j-p}, & i>j.
\end{cases}
\end{equation}
Each displayed infinite sum is a well-defined row-and-column-finite operator: when it is applied to a fixed basis vector, at most one summand acts nontrivially.

\begin{proposition}
The coefficient formula
\begin{equation}\label{eq:coef1}
\dbl e_a,e_b\dbr_1=\sum_{r\in\mathbb Z}e_r\otimes P_1(E_{ar})(e_b)
\end{equation}
is a finite sum and gives exactly the double bracket~\eqref{eq:basis1}.
\end{proposition}

\begin{proof}
Suppose first that $a\le r$. 
The first line of~\eqref{eq:P1} contributes on $e_b$ precisely when $r+p=b$ for some $p\ge1$. 
This is equivalent to $a\le r\le b-1$, and the resulting vector is $e_{a+b-r}$.
If $a>r$, the second line contributes precisely when $r-p=b$ for some $p\ge0$. 
This is equivalent to $b\le r\le a-1$, and the resulting vector is $-e_{a+b-r}$. 
These are exactly the two finite sums in~\eqref{eq:basis1}.
\end{proof}

For weight zero, define
$R_0\colon I_{\mathbb Z}\to\Endrcf(V)$ by
\begin{equation}\label{eq:R0}
R_0(E_{ij})=
\begin{cases}
\displaystyle-\sum_{p\ge0}E_{i-1-p,j-p}, & i>j,\\[0.8em]
\displaystyle\sum_{p\ge0}E_{i+p,j+1+p}, & i\le j.
\end{cases}
\end{equation}
This bilateral operator was recorded in Remark~4 of~\cite{Double-0} as the extension of the corresponding one-sided operator.

\begin{proposition}
The coefficient formula
\begin{equation}\label{eq:coef0}
\dbl e_a,e_b\dbr_0=\sum_{r\in\mathbb Z}e_r\otimes R_0(E_{ar})(e_b)
\end{equation}
is a finite sum and gives exactly the double bracket~\eqref{eq:basis0}.
\end{proposition}

\begin{proof}
If $a>r$, the first line of~\eqref{eq:R0} acts nontrivially on $e_b$ precisely when $r-p=b$ for some $p\ge0$. Thus $b\le r\le a-1$, and the resulting vector is $-e_{a+b-1-r}$. If $a\le r$, the second line acts nontrivially precisely when $r+1+p=b$ for some $p\ge0$. Thus $a\le r\le b-1$, and the resulting vector is $e_{a+b-1-r}$. This is exactly~\eqref{eq:basis0}.
\end{proof}

Note that $R_0$ and $P_1$ as operators acting from $I_{\mathbb Z}$ to $\Endrcf(V)$ satisfy the relation
$$
Q(x)Q(y) = Q(Q(x)y + xQ(y) + \lambda xy)
$$
for $\lambda = 0$ and $\lambda = 1$, respectively.
With respect to the form $\langle x,y\rangle = \tr(xy)$,
one has
$$
R_0^* = -R_0, \quad
P_1^*(x) + P_1(x) + x = \tr(x)\id_V,
$$
this means that $R_0$ and $P_1$ satisfy the infinite-dimensional analogue of $\lambda$-skew-symmetry.

\section*{Acknowledgements}

The research was carried out within the framework of the Sobolev
Institute of Mathematics state contract (project FWNF-2026-0017).

\noindent Vsevolod Gubarev \\
Sobolev Institute of Mathematics \\
Acad. Koptyug ave. 4, 630090 Novosibirsk, Russia \\
Novosibirsk State University \\
Pirogova str. 1, 630090 Novosibirsk, Russia \\
e-mail: wsewolod89@gmail.com

\end{document}